\documentclass[12pt,a4paper]{amsart}

\usepackage{graphics}
\usepackage{color}
\usepackage{amssymb}
\usepackage[all]{xy}
\usepackage{amsmath}
\usepackage{stmaryrd}

\newtheorem*{theorem}{Theorem}

\newcommand{\cc}{\mathbb C}
\newcommand{\co}{\cc^*}
\newcommand{\zz}{\mathbb Z}
\newcommand{\cone}{\operatorname{cone}}
\renewcommand{\lim}{\operatorname{lim}}

\newcommand{\T}{\mathbb T}
\newcommand{\spec}{\operatorname{Spec}}
\newcommand{\rr}{\mathbb R}

\begin{document}
\baselineskip=17pt

\title[A categorical quotient]%
{A categorical quotient in the category of dense constructible subsets}
\author[D.~Celik]{Devrim Celik (T\"ubingen)} 
\address{Mathematisches Institut, Universit\"at T\"ubingen,
Auf der Morgenstelle 10, 72076 T\"ubingen, Germany}
\email{celik@mail.mathematik.uni-tuebingen.de}
\maketitle

\begin{abstract}
A.~A'Campo-Neuen and J.~Hausen gave an example of an
algebraic torus action on an open subset of the affine
four space that admits no quotient in the category of
algebraic varieties. We show that this example admits
a quotient in the category of dense constructible subsets
and thereby answer a question of A.~Bia\l ynicki-Birula.
\end{abstract}

The purpose of this note is to 
answer a question by A.~Bia\l ynicki-Birula 
on categorical quotients in the category 
of dense constructible subsets
(dc-subsets) introduced in~\cite{BB}. 
The objects of this  category are pairs 
$X \subseteq X'$, 
where $X'$ is a complex algebraic variety 
and $X$ is a dense constructible subset
of $X'$.
A morphism from $X \subseteq X'$ to 
$Y \subseteq Y'$ is a map $X \to Y$ that 
extends to a morphism $X_1' \to Y'$ for 
some open neighbourhood 
$X_1' \subseteq X'$ of $X$.
Consider the $\cc^*$-variety
$$
X
\quad := \quad
\cc^{2}\times (\co)^{2}
\ \cup \ 
(\co)^{2}\times \cc^{2},
\qquad
tx
\quad := \quad 
(tx_{1},tx_{2},x_{3},t^{-1}x_{3}).
$$
A categorical quotient for this 
$\cc^*$-variety is a $\co$-invariant
morphism 
$\pi \colon X \to Y$, i.e. a morphism
being constant along the $\co$-orbits,
such that any other $\co$-invariant
morphism $\varphi \colon X \to Z$ 
has a unique factorization 
$\varphi = \psi \circ \pi$ with a 
morphism $\psi \colon Y \to Z$.
This concept depends strongly on 
the category one works with.
For example, in the category of 
complex algebraic 
varieties, there is no such quotient
for $X$, 
see~\cite[Prop. 5.1 (ii)]{ACHa2}, whereas in the category
of toric varieties, $X$ admits a 
quotient, namely the toric morphism
$$
\pi \colon X \ \to \ \cc^3,
\qquad 
(x_{1},x_{2},x_{3},x_{4}) 
\ \mapsto \
(x_{1}x_{4},x_{2}x_{4},x_{3}),
$$
use~\cite[Thm. 2.3]{ACHa1}.
Now, set $Y := \pi(X)$.
Then $Y \subseteq \cc^3$ is a proper dc-subset, 
and $\pi$ defines a morphism of dc-subsets
from $X \subseteq X$ to $Y \subseteq \cc^3$.
A.~Bia\l ynicki-Birula asks~\cite[p.~53]{BB},
whether or not this is a categorical quotient
in the category of dc-subsets.

\begin{theorem}
The map $\pi\colon X\rightarrow Y$ 
is a categorical quotient for the $\co$-action 
on $X$ in the category of dc-subsets. 
\end{theorem}

\begin{proof}
The most convenient way is to treat 
the problem in the setting of the 
toric varieties; we work in the notation of \cite{Fu}. 
Note that $X$ arises from the fan 
$\Delta$ in $\zz^{4}$ that has 
$\sigma_{1}:=\cone(e_{1},e_{2})$ 
and $\sigma_{2}:=\cone(e_{3},e_{4})$ 
as its maximal cones. 
Let $P\colon \zz^{4}\rightarrow \zz^{3}$ 
denote the homomorphism defined by 
\[
 P(e_{1}) \ := \ e_{1},
\qquad 
P(e_{2})\ := \ e_{2}, 
\qquad 
P(e_{3})\ := \ e_{3}, 
\qquad 
P(e_{4}) \ := \ e_{1}+e_{2}.
\]
Moreover, set $\tau_{i}:= P(\sigma_{i})$, 
$\rho_{i} := \cone(e_{i})$, where $i=1,2,3$,
and $\rho_{4}:=\cone(e_{1}+e_{2})$.
Then these cones are located in $\rr^3$ 
as indicated in the following figure.
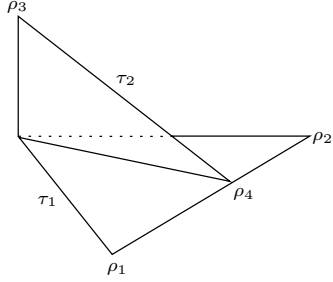
\begin{center}
\begin{picture}(0,0)%
\includegraphics{cones.pstex}%
\end{picture}%
\setlength{\unitlength}{4144sp}%
\begingroup\makeatletter\ifx\SetFigFontNFSS\undefined%
\gdef\SetFigFontNFSS#1#2#3#4#5{%
  \reset@font\fontsize{#1}{#2pt}%
  \fontfamily{#3}\fontseries{#4}\fontshape{#5}%
  \selectfont}%
\fi\endgroup%
\begin{picture}(1837,1712)(3269,-3790)
\put(4624,-3296){\makebox(0,0)[lb]{\smash{{\SetFigFontNFSS{7}{8.4}{\rmdefault}{\mddefault}{\updefault}{\color[rgb]{0,0,0}$\rho_4$}%
}}}}
\put(5091,-2962){\makebox(0,0)[lb]{\smash{{\SetFigFontNFSS{7}{8.4}{\rmdefault}{\mddefault}{\updefault}{\color[rgb]{0,0,0}$\rho_2$}%
}}}}
\put(3874,-3744){\makebox(0,0)[lb]{\smash{{\SetFigFontNFSS{7}{8.4}{\rmdefault}{\mddefault}{\updefault}{\color[rgb]{0,0,0}$\rho_1$}%
}}}}
\put(3466,-3340){\makebox(0,0)[lb]{\smash{{\SetFigFontNFSS{7}{8.4}{\rmdefault}{\mddefault}{\updefault}{\color[rgb]{0,0,0}$\tau_1$}%
}}}}
\put(3284,-2177){\makebox(0,0)[lb]{\smash{{\SetFigFontNFSS{7}{8.4}{\rmdefault}{\mddefault}{\updefault}{\color[rgb]{0,0,0}$\rho_3$}%
}}}}
\put(3928,-2618){\makebox(0,0)[lb]{\smash{{\SetFigFontNFSS{7}{8.4}{\rmdefault}{\mddefault}{\updefault}{\color[rgb]{0,0,0}$\tau_2$}%
}}}}
\end{picture}%

\end{center}
Note that $\pi\colon X\rightarrow \cc^{3}$ is the toric 
morphism defined by the homomophism $P \colon \zz^4 \to \zz^3$, 
and that the image $Y = \pi(X)$ is given by
\begin{eqnarray*}
Y
&= & 
\cc^{3}
\ \backslash \ 
(\{0\}\times \co \times \{0 \}
\ \cup \ 
\co \times \{0 \} \times \{0 \})
\\
& = & 
\T y_{0}
\ \cup \ 
\bigcup_{i=1}^{3} \T y_{\rho_{i}} 
\ \cup \ 
\T y_{\tau_{1}} 
\ \cup \ 
\T y_{\delta}
\end{eqnarray*}
where $y_0$, $y_{\rho_{i}},$ etc. in $\cc^{3}$ denote the distinguished points 
corresponding to the faces $0$, $\rho_i$ 
etc. of the cone $\delta := 
\cone(e_{1},e_{2},e_{3})$ describing 
the toric variety $\cc^{3}$.

Moreover, let $\widetilde{Y}$ denote the 
toric prevariety obtained by gluing 
the affine toric varieties 
$\widetilde{Y}_{\tau_{i}}:=\spec(\cc[\tau_{i}^{\vee}\cap \zz^{3}])$, 
where $i=1,2$, along the big torus 
$\T:=(\co)^{3} \subseteq \widetilde{Y}_{\tau_{i}}$.
Then $P \colon \zz^4 \to \zz^3$ defines 
a $\co$-invariant toric morphism
$\widetilde{\pi} \colon X \to \widetilde{Y}$.
Finally, the lattice homorphism 
$\operatorname{id}\colon \zz^{3}\rightarrow \zz^{3}$
gives us a toric morphism
$\kappa \colon \widetilde{Y} \rightarrow \cc^{3}$ 
with $\kappa(\widetilde{Y})=Y$ and 
$\pi=\kappa \circ \widetilde{\pi }$.

In order to verify the universal property 
of the morphism $\pi \colon X \to Y$, let 
$Z\subseteq Z'$ be a dc-subset and 
$\varphi \colon X \to Z$ a $\co$-invariant
morphism.
According to \cite[Prop.~5.1 (i)]{ACHa2}, 
the toric morphism 
$\widetilde{\pi} \colon X \to \widetilde{Y}$ 
is a categorical quotient for the 
$\co$-action on $X$ in the category of 
complex prevarieties.
Consequently, we have 
$\varphi= \widetilde{\psi} \circ \widetilde{\pi}$
with a morphism 
$\widetilde{\psi}\colon \widetilde{Y} \rightarrow Z'$. 
Our task is to show that there is a 
morphism $\psi \colon Y \to Z$ of dc-subsets 
making the following diagram commutative
\[
\xymatrix{
X
\ar[rr]^\varphi 
\ar[d]_{\widetilde{\pi}} 
\ar[dr]^\pi
&&
Z
\ar@{}[r]|\subseteq 
&
Z' 
\\
{\widetilde{Y}}
\ar[r]_\kappa 
\ar@/_3pc/[urrr]^{\widetilde{\psi}} 
&Y
\ar[ur]_{\psi}
&&
}
\]

In a first step, we show that such a $\psi$ 
exists as a set theoretical map. 
Since we have $\widetilde{\psi}(\widetilde{Y})=
\varphi(X)\subseteq Z$, this amounts 
to verifying that $\widetilde{\psi}$ 
is constant along the fibers of $\kappa$.
Let $\widetilde{y}_{0},\widetilde{y}_{\rho_{1}},
\widetilde{y}_{\rho_{2}},
\widetilde{y}_{\tau_{1}}\in \widetilde{Y}_{\tau_{1}}$ 
and $\widetilde{y}_{0},\widetilde{y}_{\rho_{3}},
\widetilde{y}_{\rho_{4}},
\widetilde{y}_{\tau_{2}}\in \widetilde{Y}_{\tau_{2}}$ 
denote the distinguished points corresponding 
to the cones $0$, $\rho_{i}$ etc.. 
Further denote by $\T_{y_{\tau_{i}}}$ 
the isotropy groups of $y_{\tau_{i}}$.
Then, using e.g.~\cite[Prop. 3.5]{ACHa3},
we obtain that the $\kappa$-fibres are
$$
\kappa^{-1}(ty_{0}) \ = \ t\widetilde{y}_{0}, 
\qquad
\kappa^{-1}(ty_{\rho_{i}}) \ = \ t\widetilde{y}_{\rho_{i}},
\quad 
i=1,2,3,
$$
$$
\kappa^{-1}(ty_{\tau_{1}})
\ = \ 
t\T_{y_{\tau_{1}}}\widetilde{y}_{\tau_{1}}
\cup 
t\T_{y_{\tau_{1}}}\widetilde{y}_{\rho_{4}},
\qquad
\kappa^{-1}(y_{\delta})
\ = \ 
\T \widetilde{y}_{\tau_{2}},
$$
where $t\in \T$. 
Thus, only for $\kappa^{-1}(ty_{\tau_{1}})$
and $\kappa^{-1}(y_{\delta})$, there is something 
to show.
Take $v \in \zz^{3}$ from the relative interior 
$\rho_4$ and let $\lambda_{v}\colon \co \to \T$ 
denote the corresponding one-parameter-subgroup.
Then, in $\widetilde{Y}$, we have
\begin{eqnarray*}
\overline{\lambda_{v}(\co) t \widetilde{y_{0}}}
\ \backslash \ 
\lambda_{v}(\co) t \widetilde{y_{0}}
& = & 
\{
t \widetilde{y}_{\tau_{1}}, \; t \widetilde{y}_{\rho_{4}}
\}.
\end{eqnarray*}
Note that we have
$t \widetilde{y}_{\tau_{1}}\in \widetilde{Y}_{\tau_{1}}$ 
and
$t \widetilde{y}_{\rho_{4}} \in \widetilde{Y}_{\tau_{2}}$. 
Since $Z'$ separated, we obtain
$\widetilde{\psi}(t\widetilde{y}_{\tau_{1}})=
\widetilde{\psi}(t\widetilde{y}_{\rho_{4}})$.
In particular, we obtain for 
$t\in \T_{y_{\tau_{1}}}=\T_{\widetilde{y}_{\tau_{1}}}$ 
the equation 
\[
\widetilde{\psi}(t\widetilde{y}_{\rho_{4}})
\ = \ 
\widetilde{\psi}(t\widetilde{y}_{\tau_{1}})
\ = \ 
\widetilde{\psi}(\widetilde{y}_{\tau_{1}}).
\]
Consequently, for every $t\in \T$, the set 
$\widetilde{\psi}(\kappa^{-1}(t y_{\tau_{1}}))$ 
consists of a single point.
Now we treat $\kappa^{-1}(t y_{\tau_{1}})=\T 
\widetilde{y}_{\tau_{2}}$. First note that 
$\T=\T_{\widetilde{y}_{\tau_{1}}}\T_{\widetilde{y}_{\tau_{2}}}$ 
and hence any $t\in \T$ 
can be written as  $t=t_{\tau_{1}}t_{\tau_{2}}$ 
with $t_{\tau_{i}} \in \T_{\widetilde{y}_{\tau_{i}}}$. 
Consider again the one parameter subgroup
$\lambda_{v} \colon \co \to \T$ corresponding 
to a lattice vector $v \in \zz^3$ in the 
relative interior of $\rho_4$.
Then we have:
\begin{eqnarray*}
\widetilde{\psi}(t\widetilde{y}_{\tau_{2}})
&= &
% \widetilde{\psi}(t_{\tau_{1}}t_{\tau_{2}} 
%\lim_{s\rightarrow 0}(\lambda_{v}(s)\widetilde{y}_{\rho_{4}}))\\
\lim_{s\rightarrow 0}\widetilde{\psi}( t_{\tau_{1}}t_{\tau_{2}} 
\lambda_{v}(s)
\widetilde{y}_{\rho_{4}}) 
\\
&=&
\lim_{s\rightarrow 0}\widetilde{\psi}( t_{\tau_{1}}t_{\tau_{2}} 
\lambda_{v}(s)\widetilde{y}_{\tau_{1}}) 
\\
&=&
\lim_{s\rightarrow 0}\widetilde{\psi}( t_{\tau_{2}} 
\lambda_{v}(s)\widetilde{y}_{\tau_{1}})
\\
% &=\widetilde{\psi}(t_{\tau_{2}} \lim_{s\rightarrow 0}
%(\lambda_{v}(s)\widetilde{y}_{\rho_{4}}))\\
&=&
\widetilde{\psi}(t_{\tau_{2}}\widetilde{y}_{\tau_{2}})
\\
& = &
\widetilde{\psi}(\widetilde{y}_{\tau_{2}}).
\end{eqnarray*}

We show now that the map $\psi \colon Y \to Z$ 
is a morphism of dc-subsets.
Consider the open affine toric subvarieties 
$Y_{\rho_{3}}, Y_{\tau_{1}} \subseteq \cc^{3}$ 
and $\widetilde{Y}_{\rho_{3}} \subseteq \widetilde{Y}$ 
defined by the cones $\rho_{3}$ and $\tau_{1}$. 
Then we have $Y_{\tau_{1}},Y_{\rho_{3}}\subseteq Y$ 
and, for  $U:=\widetilde{Y}_{\tau_{1}}\cup 
\widetilde{Y}_{\rho_{3}}$ and $V:=Y_{\tau_{1}}\cup Y_{\rho_{3}}$, 
the restriction $\kappa_{\vert U} \colon U\rightarrow V$ 
is an isomorphism of varieties.
Thus, we have  
$\psi=\widetilde{\psi} \circ \kappa^{-1}$ on $V$, 
and hence $\psi$ is a rational map. 

We have to show that $\psi$ defined  
near the point $y_{\delta}$. 
For this we choose an affine open 
neighbourhood $W_{0}\subseteq Z'$ 
of $\psi(y_{\delta})$ in $Z'$ 
and set $U_{0}:= \widetilde{\psi}^{-1}(W_{0})$ 
and $V_{0}:= \psi^{-1}(W_{0})$. 
Then we have an open subset 
$$
V_{0} \cap V
\ = \ 
\kappa (U_{0}\cap U) 
\ \subseteq \ 
\cc^{3}, 
$$
on which $\psi$ is a morphism.
After realizing $W_0$ as a closed subset
of some ${\mathbb K}^r$, we obtain functions 
$g_1, \ldots, g_r \in \mathcal{O}(V_{0}\cap V),$
such that the morphism $\psi$ is given on
$V_0 \cap V$ as 
\[
\psi_{\vert V_0 \cap V}
\colon
V_0 \cap V \ \to \ W_0,
\qquad
y \ \mapsto \ (g_1(y), \ldots, g_r(y)).
\] 

Suppose for the moment that each $g_i$ is defined 
at $y_{\delta}$.
Then there are an open set 
$V' \subseteq \cc^{3}$ containing $V_0 \cap V$
and the point $y_{\delta}$ such that 
$g_1, \ldots, g_r$ define a morphism
$\psi' \colon V' \to W_0$ extending 
$\psi_{\vert V_0 \cap V}$.
This morphism fits into the diagram
\[ 
\xymatrix{
U_0 \cap U
\ar@{}[r]|\subseteq 
\ar[d]_{\kappa}
&
{\kappa^{-1}(V')}
\ar[r]^{\widetilde{\psi}}
\ar[d]_{\kappa}
&
W_0
\\
V_0 \cap V
\ar@{}[r]|\subseteq 
&
V'
\ar[ur]_{\psi'}
& 
}
\]
Since $\widetilde{\psi}$ and $\psi' \circ \kappa$ coincide 
on the dense open subset $U_{0} \cap U$, 
they coincide on $\kappa^{-1}(V')$.
Thus, by surjectivity of $\kappa$, we obtain
$\psi' = \psi_{\vert V'}$ and see that $\psi$ 
is a morphism on the neighbourhood $V' \subseteq \cc^{3}$ 
of $y_{\delta} \in Y$. Hence, $\psi$ is a morphism
on the open set $V\cup V'\subseteq \cc^{3}$. 
This shows that $\psi \colon Y \rightarrow Z$
is a mophism in the category of dc-subsets.

Thus, our task is to show that every $g_i$ 
is defined in $y_{\delta}$.
For this, we regard $g_i$ as a rational function on 
the normal ambient variety $\cc^{3}$ of $Y$ 
and show that $\operatorname{div}(g_i)$ is nonnegative
at $y_{\delta}$.
For this in turn, it suffices to show that any prime 
divisor on $\cc^{3}$ containing $y_{\delta}$ meets the 
open set $V \cap V_0$.

Let $D$ be a prime divisor on $\cc^{3}$ passing through 
$y_{\delta}$. 
We show that $D$ meets the open set 
$V \cap V_0$. Since $\cc^{3}\setminus V$ is of 
codimension 2 in $\cc^{3}$, we obtain 
$D\cap V \ne \emptyset$. First consider the case 
that $\kappa^{-1}(D)$ contains the prime divisor 
$\overline{\T\widetilde{y}_{\rho_{4}}}
 = \T\widetilde{y}_{\rho_{4}}\cup \T \widetilde{y}_{\tau_{2}}$. 
Then $\kappa^{-1}(D)\cap U_{0}\ne \emptyset$, hence 
\begin{eqnarray*} 
D \cap V_{0} 
\ \ne \ 
\emptyset
& \implies & 
D \cap V \cap  V_0 
\ \ne \ 
\emptyset.
\end{eqnarray*}
Next suppose that $\kappa^{-1}(D)$ 
doesn't contain $\overline{\T 
\widetilde{y}_{\rho_{4}}}=
 \widetilde{Y}\setminus U$.
Then any component of $\kappa^{-1}(D)$
meets $U$. Since $\cc^{3}$ is factorial, 
$D$ is principal, and thus 
$\kappa^{-1}(D)$ is of pure codimension
one in $\widetilde{Y}$.
Consequently, $\kappa^{-1}(D)$ equals the closure of 
$\kappa^{-1}(D) \cap U$ in $\widetilde{Y}$. 
Since $\kappa^{-1}(D)$ intersects $\kappa^{-1}(y_{\delta})$, 
we can conclude
\begin{eqnarray*} 
\kappa^{-1}(D) \cap U \cap  U_0 
\ \ne \ 
\emptyset
& \implies & 
D \cap V \cap  V_0 
\ \ne \ 
\emptyset.
\end{eqnarray*}
\end{proof}

\end{document}